\documentclass[11pt,a4paper]{article}
\usepackage{amsmath}
\usepackage{amsfonts}
\usepackage{amssymb}
\usepackage{graphicx}
\usepackage{url}

\newtheorem{thm}{Theorem}[section]
\newtheorem{lemma}[thm]{Lemma}

\newcommand{\ICG}{\mathrm{ICG}}

\newenvironment{proof} {\par \noindent \textbf{Proof: }}{\QED \par \bigskip }
\newcommand{\QED}{\hfill$\square$}

\newcommand{\N}{{\mathbb N}}

\title {
\bf{New results on the energy of integral circulant graphs} }

\author {
{\large Aleksandar Ili\' c}  \\
{\em Faculty of Sciences and Mathematics, Vi\v segradska 33, 18000 Ni\v s, Serbia} \\
{e-mail: \texttt{aleksandari@gmail.com} } \and
{\large Milan Ba\v si\' c} \\
{\em Faculty of Sciences and Mathematics, Vi\v segradska 33, 18000 Ni\v s, Serbia} \\
{e-mail: \texttt{basic\_milan@yahoo.com} }}

\begin{document}

\maketitle

\begin{abstract}
Circulant graphs are an important class of interconnection networks
in parallel and distributed computing. Integral circulant graphs
play an important role in modeling quantum spin networks supporting
the perfect state transfer as well. The integral circulant graph
$\ICG_n (D)$ has the vertex set $Z_n = \{0, 1, 2, \ldots, n - 1\}$
and vertices $a$ and $b$ are adjacent if $\gcd(a-b,n)\in D$, where
$D \subseteq \{d : d \mid n,\ 1\leq d<n\}$. These graphs are highly
symmetric, have integral spectra and some remarkable properties
connecting chemical graph theory and number theory. The energy of a
graph was first defined by Gutman, as the sum of the absolute values
of the eigenvalues of the adjacency matrix. Recently, there was a
vast research for the pairs and families of non-cospectral graphs
having equal energies. Following [R. B. Bapat, S. Pati,
\textit{Energy of a graph is never an odd integer}, Bull. Kerala
Math. Assoc. 1 (2004) 129--132.], we characterize the energy of
integral circulant graph modulo 4. Furthermore, we establish some
general closed form expressions for the energy of integral circulant
graphs and generalize some results from [A. Ili\' c, \textit{The
energy of unitary Cayley graphs}, Linear Algebra Appl. 431 (2009),
1881--1889.]. We close the paper by proposing some open problems and
characterizing extremal graphs with minimal energy among integral
circulant graphs with $n$ vertices, provided $n$ is even.
\end{abstract}

 \vskip
0.1cm

 \vskip 0.1cm

\section{Introduction}

Circulant graphs are Cayley graphs over a cyclic group. The interest of circulant graphs in graph
theory and applications has grown during the last two decades, they appeared in coding theory, VLSI
design, Ramsey theory and other areas. Recently there is vast research on the interconnection
schemes based on circulant topology -- circulant graphs represent an important class of
interconnection networks in parallel and distributed computing (see \cite{hwang03}). Integral
circulant graphs are also highly symmetric and have some remarkable properties between connecting
graph theory and number theory.

In quantum communication scenario, circulant graphs is used 
in the problem of arranging $N$ interacting qubits in a quantum spin
network based on a circulant topology
 to obtain good communication between them. In general,
quantum spin system can be defined as a collection of qubits on a
graph, whose dynamics is governed by a suitable Hamiltonian, without
external control on the system.
Different classes of graphs were examined for the purpose of perfect
transferring the states of the systems. Since circulant graphs are
mirror symmetric, they represent good candidates for the property of
periodicity and thus integrality \cite{fizicarski1}, which further
implies that
integral circulant graphs would be potential candidates for
modeling the quantum spin networks that permit perfect state
transfer \cite{ahmadi,sandiego,sandiego1,Go08,SaSeSh07}. These
properties are primarily related to the spectra of these graphs.
Indeed, the eigenvalues of the graphs are indexed in palindromic
order ($\lambda_i=\lambda_{n-i}$) and can be represented by
Ramanujan's sums.

Ba\v si\'c \cite{stevanovic08a,Ba10} established a condition under
which integral circulant graphs have perfect state transfer and gave
complete characterization these graphs. It turned out that the
degree of $2$ must be equal in a prime factorization of the
difference of successive eigenvalues. Furthermore, exactly one of
the divisors $n/4$ or $n/2$ have to belong to the divisor set $D$
for any integral circulant graph $\ICG_n(D)$ having perfect state
transfer. In this paper we continue with studying parameters of
integral circulant graphs like energy, having in mind application in
chemical graph theory. We actually  focus on characterization of the
energy of integral circulant graphs $\ICG_n(D)$  modulo $4$, where
the divisor $n/2$ and eigenvalue $\lambda_{n/2}$ play important
role. During this task, some interesting properties of the
eigenvalues modulo $2$ are also used.

Saxena, Severini and Shraplinski \cite{SaSeSh07} studied some parameters of integral circulant
graphs as the bounds for the number of vertices and the diameter, bipartiteness and perfect state
transfer. The present authors in \cite{BaIl09,IlBa09} calculated the clique and chromatic number of
integral circulant graphs with exactly one and two divisors, and also disproved posed conjecture
that the order of $\ICG_n (D)$ is divisible by the clique number. Klotz and Sander \cite{klotz07}
determined the diameter, clique number, chromatic number and eigenvalues of the unitary Cayley
graphs. The latter group of authors proposed a generalization of unitary Cayley graphs named {\it
gcd-graphs} and proved that they have to be integral.

Let $A$ be the adjacency matrix of a simple graph $G$, and
$\lambda_1, \lambda_2, \ldots, \lambda_n$ be the eigenvalues of the
graph $G$. The energy of $G$ is defined as the sum of absolute
values of its eigenvalues \cite{Gu78,Gu01,E}
$$
E(G) = \sum_{i = 1}^n |\lambda_i|.
$$

The concept of graph energy arose in chemistry where certain numerical quantities, such as the heat
of formation of a hydrocarbon, are related to total $\pi$-electron energy that can be calculated as
the energy of an appropriate molecular graph.

The graph $G$ is said to be hyperenergetic if its energy exceeds the
energy of the complete graph $K_n$, or equivalently if $E (G) > 2n -
2$. This concept was introduced first by Gutman and afterwards has
been studied intensively in the literature
\cite{AkMoZa09,BlSh08,Gu99,So09}. Hyperenergetic graphs are
important because molecular graphs with maximum energy pertain to
maximality stable $\pi$-electron systems. In \cite{Il09} and
\cite{RaVe09}, the authors calculated the energy of unitary Cayley
graphs and complement of unitary Cayley graphs, and establish the
necessary and sufficient conditions for $\ICG_n$ to be
hyperenergetic. There was a vast research for the pairs and families
of non-cospectral graphs having equal energy
\cite{BoViAb08,BrStGu04,IlBaGu10,Il10,C,B,RaWa07,A}.

In 2004 Bapat and Pati \cite{BaPa04} proved an interesting simple
result -- the energy of a graph cannot be an odd integer. Pirzada
and Gutman \cite{PiGu08} generalized this result and proved the
following
\begin{thm}
Let $r$ and $s$ be integers such that $r \geq 1$ and $0 \leq s \leq r - 1$. Let $q$ be an odd
integer. Then $E (G)$ cannot be of the form $(2^s q)^{1/r}$.
\end{thm}
For more information about the closed forms of the graph energy we
refer the reader to \cite{D}.

\smallskip

In this paper we go a step further and characterize the energy of integral circulant graph
modulo~4.

The paper is organized as follows. In Section 2 we give some
preliminary results regarding eigenvalues of integral circulant
graphs. In Section 3 we characterize the energy of integral
circulant graph modulo 4, while in Section 4 we generalized formulas
for the energy of integral
circulant graphs from \cite{Il09}. 
In Section 5, some larger families of graphs with equal energy are
presented and further we support conjecture proposed by So
\cite{so06}, that two graphs $\ICG_n(D_1)$ and $\ICG_n(D_2)$ are
cospectral if and only if $D_1=D_2$. In concluding remarks we
propose some open problems and characterize extremal graphs with
minimal energy among integral circulant graphs with $n$ vertices,
provided $n$ is even.

\section{Preliminaries}

Let us recall that for a positive integer $n$ and subset $S
\subseteq \{0, 1, 2, \ldots, n - 1\}$, the circulant graph $G(n, S)$
is the graph with $n$ vertices, labeled with integers modulo $n$,
such that each vertex $i$ is adjacent to $|S|$ other vertices $\{ i
+ s \pmod n \ | \ s \in S\}$. The set $S$ is called a symbol of
$G(n, S)$. As we will consider only undirected graphs without loops,
we assume that $0\not\in S$ and, $s \in S$ if and only if $n - s \in
S$, and therefore the vertex $i$ is adjacent to vertices $i \pm s
\pmod n$ for each $s \in S$.

Recently, So \cite{so06} has characterized circulant graphs with
integral eigenvalues-integral circulant graphs. Let
$$ G_n(d) = \{ k\ | \ \gcd(k, n)=d, \ 1 \leq k < n \} $$
be the set of all positive integers less than $n$ having the same
greatest common divisor $d$ with $n$. Let $D_n$ be the set of
positive divisors $d$ of $n$, with $d \leq \frac{n}{2}$.

\begin{thm}
A circulant graph $G(n,S)$ is integral if and only if
$$ S = \bigcup_{d \in D} G_n(d) $$
for some set of divisors $D \subseteq D_n$. \label{so}
\end{thm}

We denote them by $\ICG_n(D)$ and in some recent papers integral
circulant graphs are also known as gcd-graphs
(\cite{BaIl09,klotz07}).

Let $\Gamma$ be a multiplicative group with identity $e$. For
$S\subset \Gamma$, $e\not\in S$ and $S^{-1} = \{s^{-1}\ |\ s\in
S\}=S$, the Cayley graph $X = Cay(\Gamma,S)$ is the undirected
graph having vertex set $V(X)=\Gamma$ and edge set $E(X) =
\{\{a,b\}\ |\ ab^{-1}\in S\}$. For a positive integer $n > 1$ the
unitary Cayley graph $X_n = Cay(Z_n, \ U_n)$ is defined by the
additive group of the ring $Z_n$ of integers modulo $n$ and the
multiplicative group $U_n = Z_n^{*}$ of its invertible elements.


 By Theorem
\ref{so} we obtain that integral circulant graphs are Cayley
graphs of the additive group of $Z_n$ with respect to the Cayley
set $S = \bigcup_{d\in D}G_n(d)$.
From Corollary 4.2 in \cite{hwang03}, the graph
$\ICG_n(D)$ is connected if and only if
$\gcd(d_1,d_2,\ldots,d_k)=1$.

Let $A$ be a circulant matrix. The entries $a_0, a_1, \ldots, a_{n -
1}$ of the first row of the circulant matrix $A$ generate the
entries of the other rows by a cyclic shift (for more details see
\cite{Da79}). There is an explicit formula for the eigenvalues
$\lambda_k$, $0 \leqslant k \leqslant n - 1$, of a circulant matrix
$A$. Define the polynomial $P_n (z)$ by the entries of the first row
of $A$,
$$
P_n (z) = \sum_{i = 0}^{n - 1} a_i \cdot z^i
$$
The eigenvalues of $A$ are given by
\begin{equation}
\label{eq:ramanujan} \lambda_j = P_n (\omega^j) = \sum_{i = 0}^{n -
1} a_i \cdot \omega^{j i}, \qquad 0 \leqslant j \leqslant n - 1,
\end{equation}
where $\omega=\exp(\i2\pi/n)$ is the $n$-th root of unity.
Ramanujan's sum \cite{Wiki}, usually denoted $c (k, n)$, is a
function of two positive integer variables $n$ and $k$ defined by
the formula
$$
c (k, n)= \sum_{a = 1 \atop \gcd(a,n)=1}^n e^{\frac{2 \pi i}{n} \cdot a k} = \sum_{a=1 \atop
\gcd(a,n)=1}^n \omega_n^{a k},
$$
where $\omega_n$ denotes a complex primitive $n$-th root of unity. These sums take only integral
values,
$$
c (k, n) = \mu \left (t_{n,k} \right ) \cdot \frac{ \varphi (n) }
{\varphi \left (t_{n,k} \right )} \quad \mbox{where} \quad
t_{n,k}=\frac{n}{\gcd (k, n)},
$$
and $\mu$ denotes the M\" obious function. In \cite{klotz07} it was proven that gcd-graphs (the
same term as integral circulant graphs $\ICG_n (D)$) have integral spectrum,
\begin{equation}
\label{eq:eigenvalues} \lambda_k = \sum_{d \in D} c \left(k, \frac{n}{d}\right), \qquad 0 \leqslant
k \leqslant n - 1.
\end{equation}

Using the well-known summation \cite{HaWr80}
$$
s (k, n) = \sum_{i = 0}^{n - 1} \omega_n^{i k} = \left\{
\begin{array}{l l}
  0 & \quad \mbox{ if } \quad n \nmid k\\
  n & \quad \mbox{ if } \quad n \mid k\\
\end{array} \right.,
$$
we get that
\begin{equation}
\label{eq:sum} \sum_{k = 0}^{n - 1} c (k, n) = 0.
\end{equation}
For even $n$ it follows
\begin{eqnarray}
\label{eq:sum1} \sum_{k = 0}^{ n / 2 - 1} c (k, n) &=& \sum_{a =1
\atop \gcd(a,n)=1}^n \sum_{k = 0}^{n/2 - 1} \omega_n^{a k} = \sum_{a
= 1 \atop \gcd(a,n)=1}^{n / 2} \left ( \sum_{k = 0}^{n/2 - 1}
\omega_n^{a k} + \omega_n^{(n - a)k} \right) \nonumber \\  &=&
\sum_{a = 1 \atop \gcd(a,n)=1}^{n / 2} \left ( \omega_n^{an} -
\omega_n^{an/2} + \sum_{k = 0}^{n-1} \omega_n^{a k} \right) =
\sum_{a = 1
\atop \gcd(a,n)=1}^{n / 2} (1 + 1) \nonumber\\
&=& \varphi (n).
\end{eqnarray}
Similarly, for odd $n$ it follows
\begin{equation}
\label{eq:sum2} \sum_{k = 0}^{ (n - 1) / 2 } c (k, n) =
\frac{\varphi (n)}{2}.
\end{equation}

It also follows that if $k \equiv k' \pmod n$ then $c (k, n) = c (k', n)$.

Throughout the paper, we let $n = p_1^{\alpha_1} p_2^{\alpha_2}
\cdot \ldots \cdot p_k^{\alpha_k}$, where $p_1 < p_2 < \ldots <
p_k$ are distinct primes, and $\alpha_i \geq 1$.

\section{The energy of integral circulant graphs modulo 4}

Note that for arbitrary divisor $d$ and $1 \leq i \leq n - 1$, it holds
$$
t_{n / d, i} = \frac{n/d}{\gcd (n/d, i)} = \frac{n}{\gcd (n, id)}
$$
and
$$
t_{n / d, n - i} = \frac{n/d}{\gcd (n/d, n - i)} = \frac{n}{\gcd (n, nd - id)}.
$$
Since $\gcd (n, id) = \gcd (n,nd - id)$, we have $t_{n / d, i} =
t_{n / d, n - i}$. Finally,
$$
c (i,n / d) = \mu (t_{n / d, i}) \frac{ \varphi (n / d) }{ \varphi
(t_{n / d, i})} = \mu (t_{n / d, n - i}) \frac{ \varphi (n / d) }{
\varphi (t_{n / d, n - i})} = c (n - i,n / d),
$$
for each $1 \leq i \leq n - 1$. Therefore we have the following
assertion.

\begin{lemma}
\label{le-cool} Let $\ICG_n (D)$ be an arbitrary integral circulant graph. Then for each $1 \leq i
\leq n - 1$, the eigenvalues $\lambda_i$ and $\lambda_{n - i}$ of $\ICG_n (D)$ are equal.
\end{lemma}

For $i = 0$ we have
$$
\lambda_0 = \sum_{d \in D} \varphi (n / d),
$$
while for $n$ even and $i = n / 2$ we have
$$
\lambda_{n / 2} = \sum_{d \in D} (-1)^d \varphi (n / d).
$$

\subsection{Energy modulo 4 for $n$ odd}

According to Lemma \ref{le-cool}, the energy of $G \cong \ICG_n (D)$ is equal to
$$
E (G) = \lambda_0 + 2 \sum_{i = 1}^{(n - 1) / 2} |\lambda_i|.
$$
Since $x \equiv |x| \pmod 2$, in order to characterize $E (G)$ modulo $4$ we consider the parity of
the following sum
$$
\frac {E (G)} 2 \equiv \sum_{d \in D} \frac {\varphi (n / d)} 2+
\sum_{i = 1}^{(n - 1) / 2} \sum_{d \in D} c (i, n / d) \pmod 2.
$$
Since $n/d > 2$, it follows that $\varphi (n/d)$ is even. After
exchanging the order of the summation we have
\begin{equation}
\label{eq:energy_modulo 2} \frac {E (G)} 2 \equiv \sum_{d \in D} \frac {\varphi (n / d)} 2 +
\sum_{d \in D} \sum_{i = 1}^{(n - 1) / 2} c (i, n / d) \pmod 2.
\end{equation}

By relation \eqref{eq:sum}, we get that for every $k$ it holds that
\begin{equation}
\label{eq:gen} \sum_{i = k}^{k + n - 1} c (i, n) = 0
\end{equation}

\begin{thm}
For odd $n$, the energy of $\ICG_n(D)$ is divisible by four.
\end{thm}

\begin{proof}
Using the following relation $\frac {n-1} 2=\frac n d\cdot \frac{d-1} 2+\frac {n-d}{2d}$, the
formula for graph energy (\ref{eq:energy_modulo 2}) now becomes
\begin{equation}
\label{eq:energy_n_odd}
 \frac {E (G)} 2 \equiv \sum_{d \in D} \frac {\varphi (n / d)} 2
+ \sum_{d \in D} \left(\sum_{l =0 }^{\frac {d - 1} 2-1}\sum_{i=l\frac n d+1}^{(l+1)\frac n d} c (i,
n / d)+\sum_{i=\frac{n(d-1)}{2d}+1}^{\frac {n-1} 2} c (i, n / d)\right)\pmod 2.
\end{equation}

Next we get
$$
\frac {E (G)} 2 \equiv \sum_{d \in D} \frac {\varphi (n / d)} 2 + \sum_{i=1}^{(n/d - 1) / 2} c (i,
n / d)\pmod 2,
$$
and using relation \eqref{eq:sum2}, we get that
$$
\frac {E (G)} 2 \equiv \sum_{d \in D} \frac {\varphi (n / d)} 2 +
\frac {\varphi (n / d)} 2 -\varphi (n / d) \equiv 0 \pmod 2.
$$
This implies that $4\mid E(G)$.
\end{proof}

\subsection{Energy modulo 4 for $n$ even}

According to Lemma \ref{le-cool}, the energy of $G \cong \ICG_n (D)$ is equal to
$$
E (G) = |\lambda_0| + |\lambda_{n / 2}| + 2 \sum_{i = 1}^{n/2 - 1} |\lambda_i|.
$$
Using the same reasoning as in the previous subsection, we get that $\lambda_0$ and $\lambda_{n/2}$
are of the same parity,
$$
|\lambda_0| + |\lambda_{n/2}| = \sum_{d \in D} \varphi (n / d) + \left | \sum_{d \in D} (-1)^d
\varphi (n / d) \right|.
$$
Also,
\begin{eqnarray*}
S &=& \frac{1}{2} \cdot \left (|\lambda_0| + |\lambda_{n/2}| \right ) = \sum_{d \in D}
\frac{\varphi (n / d)}{2} + \left |\sum_{d \in D} (-1)^d \frac{\varphi (n / d)}{2} \right| \\ &=&
\left\{\begin{array}{ll}
\sum_{d \in D, \ d \ even} \varphi (n / d) , & \mbox{if $\lambda_{n/2} > 0$}\\
\sum_{d \in D, \ d \ odd} \varphi (n / d) , & \mbox{if $\lambda_{n/2} < 0$}\\
\end{array}\right..
\end{eqnarray*}

If $\frac{n}{2} \not \in D$, then $2 \mid \varphi (n/d)$ and $S
\equiv 0 \pmod 2$; otherwise we conclude that
$$
S \equiv \left\{\begin{array}{ll}
0, & \mbox{if $\lambda_{n/2} > 0$ and $4 \nmid n$, or $\lambda_{n/2} < 0$ and $4 \mid n$ } \pmod 2\\
1, & \mbox{if $\lambda_{n/2} > 0$ and $4 \mid n$, or $\lambda_{n/2} < 0$ and $4 \nmid n$ } \pmod 2\\
\end{array}\right..
$$
Therefore
$$
\frac {E (G)} 2 \equiv S + \sum_{d \in D} \sum_{i = 1}^{n / 2 - 1} c (i, n / d) \pmod 2.
$$

\begin{thm}
For even $n$, the energy of $\ICG_n(D)$ is not divisible by four if
and only if $\frac{n}{2} \not \in D$ and $\lambda_{n/2}$ is
negative.
\end{thm}

\begin{proof}
If $d$ is even, we have $\frac{n}{2} - 1 = \frac{d}{2} \cdot
\frac{n}{d} - 1$. Since $c (0, n / d) = \varphi (n / d)$, it follows
\begin{eqnarray*}
\sum_{i = 1}^{n / 2 - 1} c (i, n / d) &=& -c (0, n/d) + \sum_{k = 1}^{d / 2} \sum_{i = (k - 1)
\cdot n / d}^{k \cdot n / d - 1} c (i, n / d) \\
&=& - \varphi (n/d) + \frac{d}{2} \cdot \sum_{i = 0}^{n / d - 1} c (i, n / d)\\
&=& - \varphi (n/d).
\end{eqnarray*}

If $d$ is odd, we have $\frac{n}{2} - 1 = \frac{d - 1}{2} \cdot
\frac{n}{d} + \frac{1}{2} \cdot \frac{n}{d} - 1$. Similarly, using
the relation \eqref{eq:sum1}, it follows
\begin{eqnarray*}
\sum_{i = 1}^{n / 2 - 1} c (i, n / d) &=& -c (0, n/d) + \sum_{k = 1}^{(d - 1) / 2} \sum_{i = (k -
1) \cdot n / d}^{k \cdot n / d - 1} c (i, n / d) + \sum_{i = ((d - 1)/2) \cdot n / d}^{n/2 - 1} c
(i, n / d)
\\
&=& -\varphi (n/d) + \frac{d - 1}{2} \cdot \sum_{i = 0}^{n / d - 1} c (i, n / d) +  \sum_{i = 0}^{n
/ (2d) - 1} c (i, n / d) \\
&=& - \varphi (n / d) + \varphi (n /d) = 0.
\end{eqnarray*}

For $\frac{n}{2}\not\in D$, we have that $S \equiv 0 \pmod 2$ and $4
\mid E (G)$.

For $\frac{n}{2}\in D$, by combining above cases we have
$$
\sum_{d\in D}\sum_{i = 1}^{n / 2 - 1} c (i, n / d) \equiv \frac{1 +
(-1)^{n/2}}{2} \pmod 2.
$$
For $\lambda_{n/2} > 0$, it follows
$$
\frac {E (G)}{2} \equiv S + \frac{1 + (-1)^{n/2}}{2} \equiv \frac{1 + (-1)^{n/2}}{2} + \frac{1 +
(-1)^{n/2}}{2} \equiv 0 \pmod 2,
$$
while for $\lambda_{n/2} < 0$, we have
$$
\frac {E (G)}{2} \equiv S + \frac{1 + (-1)^{n/2}}{2} \equiv \frac{1 - (-1)^{n/2}}{2} + \frac{1 +
(-1)^{n/2}}{2} \equiv 1 \pmod 2.
$$

This completes the proof.
\end{proof}

\section{The energy of some classes of integral circulant graphs}

Here we generalize results from \cite{Il09}.

\begin{thm}
Let $n\geq 4$ be an arbitrary integer. Then the energy of the integral circulant graph
$X_n(1,p^\gamma)$ for $\gamma\geq 1$ is given by
\begin{equation}
 \aligned E(X_n(1,p^\gamma))=
    \left\{\begin{array}{rl} 2^{k-1} (\varphi (n) + \varphi (n / p)), & p \| n\\
                             2^{k-1} (2\varphi (n) + (p^{\gamma}-2p+2) \varphi (n / p)),& p^\gamma \| n, \ \gamma\geq 2\\
                             2^{k} (\varphi (n) + (p^{\gamma}-p+1) \varphi (n / p)), & p^\gamma \not\| n.
                              \end{array}
                              \right.
\endaligned
\end{equation}
\end{thm}

\begin{proof}
Let $p=p_s$ and $\gamma=\gamma_s$, where $1 \leq s \leq k$. Let $j = {p_1}^{\beta_1} p_2^{\beta_2}
\cdot \ldots \cdot p_k^{\beta_k}\cdot J$ be a representation of an arbitrary index $0\leq j\leq
n-1$, where $gcd(J,n)=1$. The $j$-th eigenvalue of $X_n (1, p_s^{\gamma_s})$ is given by
$$
\lambda_j=c(j,n)+c(j,n/p_s^{\gamma_s}).
$$

Suppose that there exists a prime number $p_i \mid j$ for some $i \neq s$ such $\beta_i \leq
\alpha_i - 2$. This implies that $p_i^2\mid t_{n,j}$ and $p_i^2\mid t_{n/p_s^{\gamma_s},j}$.
Furthermore, we have $\mu(t_{n,j})=\mu(t_{n/p_s^{\gamma_s},j}) = 0$ and thus $\lambda_j =
0$.\smallskip

If $\beta_s\leq \alpha_s-\gamma_s-1$ then $p_s^3 \mid t_{n,j}$ and
$p_s^2 \mid t_{n/p_s^{\gamma_s},j}$. Similarly, we conclude that
$\lambda_j=0$.
\smallskip

For an arbitrary index $j$, define the set $P=\{1\leq i\leq k\ |\
i \neq s, \beta_i=\alpha_i-1\}$.

Let $J_l=\{0\leq j\leq n-1\ |\ \beta_s=\alpha_s-l,\ \alpha_i-1\leq
\beta_i\leq \alpha_i\ \mbox{ for}\ i \neq s\}$, for $0\leq l\leq
\gamma_s+1$.


\noindent {\bf Case 1.} For $l=0$ and $j\in J_0$ we have
$$
t_{n,j}=\frac n {\gcd(j,n)}=
\frac {{p_1}^{\alpha_1}p_2^{\alpha_2} \cdots p_s^{\alpha_s} \cdots p_k^{\alpha_k}}
{p_1^{\beta_1}p_2^{\beta_2} \cdots p_s^{\alpha_s} \cdots p_k^{\beta_k}} = \prod_{i\in P}p_i.
$$

On the other hand, it follows
$$
t_{n/p_s^{\gamma_s},j}=\frac {n/p_s^{\gamma_s}} {gcd(j,n/p_s^{\gamma_s})}= \frac
{{p_1}^{\alpha_1}p_2^{\alpha_2} \cdots p_s^{\alpha_s - \gamma_s} \cdots p_k^{\alpha_k}}
{p_1^{\beta_1}p_2^{\beta_2} \cdots p_s^{\alpha_s - \gamma_s} \cdots p_k^{\beta_k}} = \prod_{i\in P}
p_i.
$$

The $j$-th eigenvalue is given by
$$
\lambda_j=c(j,n)+c(j,n/p_s^{\gamma_s})=(-1)^{|P|}\frac {\varphi(n)}{\varphi(\prod_{i\in
P}p_i)}+(-1)^{|P|}\frac {\varphi(n/p_s^{\gamma_s})}{\varphi(\prod_{i\in P}p_i)}= (-1)^{|P|}\frac
{\varphi(n)+\varphi(n/p_s^{\gamma_s})}{\varphi(\prod_{i\in P}p_i)}.
$$

The number of indices  $j\in J_0$ with the same set $P$ is equal to
the number of $J$ such that
$$\gcd \left (J, \frac n {p_1^{\beta_1}p_2^{\beta_2} \cdots p_s^{\alpha_s} \cdots
p_k^{\beta_k}} \right )=1.$$
The last equation implies that the number of such indices is equal to
the Euler's totient function $$\varphi(\frac n {p_1^{\beta_1}p_2^{\beta_2} \cdots p_s^{\alpha_s} \cdots
p_k^{\beta_k}})=\varphi(\prod_{i\in P}p_i).$$

\noindent {\bf Case 2.} Let $l=1$ 
 and for $j\in J_1$ we similarly obtain
$t_{n,j}=p_s \prod_{i\in P} p_i$ and
$t_{n/p_s^{\gamma_s},j}=\prod_{i\in P}p_i$. Therefore, the $j$-th
eigenvalue is given by
$$
\lambda_j=(-1)^{|P|+1}\frac {\varphi(n)}{\varphi(p_s \prod_{i\in
P}p_i)}+(-1)^{|P|}\frac
{\varphi(n/p_s^{\gamma_s})}{\varphi(\prod_{i\in P}p_i)}=
(-1)^{|P|}\frac{(-\varphi(n)+(p_s-1)\varphi(n/p_s^{\gamma_s}))}{(p_s-1)
\varphi(\prod_{i\in P}p_i)}.
$$

The number of indices $j\in J_1$ with the same set $P$ is equal to
$$
\varphi(\frac n {p_1^{\beta_1}p_2^{\beta_2} \cdots p_s^{\alpha_s - 1} \cdots
p_k^{\beta_k}})=\varphi(p_s \prod_{i\in P}p_i) = (p_s - 1) \varphi (\prod_{i\in P}p_i).
$$

\noindent {\bf Case 3.} For $2\leq l\leq \gamma_s$ 
 and  $j\in J_l$ we obtain $p_s^{\alpha_s-\gamma_s-\min (\alpha_s-l,\alpha_s-\gamma_s)}\| t_{n/p_s^{\gamma_s},j}$
 which implies that $p_s\nmid t_{n/p_s^{\gamma_s},j}$ and $t_{n/p_s^{\gamma_s},j}=\prod_{i\in P}p_i$.
Since $p_s^l \mid t_{n,j}$ and $l\geq 2$ it holds that
$\mu(t_{n,j})= 0$. Therefore, the $j$-th eigenvalue is given by
$$
\lambda_j=(-1)^{|P|}\frac{\varphi(n/p_s^{\gamma_s})}{\varphi(
\prod_{i\in P}p_i)}=(-1)^{|P|}\frac
{\varphi(n/p_s^{\gamma_s})}{\varphi(\prod_{i\in P}p_i)}.
$$

The number of indices $j\in J_l$ with the same set $P$ is equal to
$$\varphi(\frac n {p_1^{\beta_1}p_2^{\beta_2} \cdots p_s^{\alpha_s - l} \cdots
p_k^{\beta_k}})= \varphi(p_s^l \prod_{i\in P}p_i) = p_s^{l-1} (p_s -
1) \varphi(\prod_{i\in P}p_i).$$

\smallskip

\noindent {\bf Case 4.} For $l=\gamma_s+1$ and $j\in J_{\gamma_s+1}$ we obtain
$p_s^{\gamma_s+1}\|t_{n,j}$ and $c(j,n)=\mu(t_{n,j})= 0$. Also, it holds that
$p_s^{\alpha_s-\gamma_s-\min (\alpha_s-\gamma_s-1,\alpha_s-\gamma_s)}\| t_{n/p_s^{\gamma_s},j}$
which yields that $p_s\|t_{n/p_s^{\gamma_s},j}$.  Therefore, the $j$-th eigenvalue is given by
$$
\lambda_j=(-1)^{|P|+1}\frac{\varphi(n/p_s^{\gamma_s})}{\varphi(
p_s\prod_{i\in P}p_i)}=(-1)^{|P|+1}\frac
{\varphi(n/p_s^{\gamma_s})}{\varphi(p_s\prod_{i\in P}p_i)}.
$$

The number of indices $j\in J_l$ with the same set $P$ is equal to
$$\varphi(\frac n {p_1^{\beta_1}p_2^{\beta_2} \cdots p_s^{\alpha_s - \gamma_s-1} \cdots
p_k^{\beta_k}})= \varphi(p_s^{\gamma_s+1} \prod_{i\in P}p_i) =
p_s^{\gamma_s} (p_s - 1) \varphi(\prod_{i\in P}p_i).$$

\smallskip


After all mention cases, the energy of $X_n(1, p_s^{\gamma_s})$ is
given by
\begin{eqnarray}
\label{eq:energy} E(X_n(1,p_s^{\gamma_s}))&=&\sum_{j=0}^{n-1}|\lambda_j|\\ \nonumber
&=&\sum_{P\subseteq\{1,2,\ldots,k\}\setminus \{s\}}\left ( 
\frac {\varphi(n)+\varphi(n/p_s^{\gamma_s})}{\varphi(\prod_{i\in P}p_i)}\cdot
\varphi(\prod_{i\in P}p_i) \right.\\ \nonumber && 
+ \frac {\varphi(n)-(p_s - 1) \varphi(n/p_s^{\gamma_s})}{(p_s-1) \varphi(\prod_{i\in P}p_i)}\cdot
(p_s-1) \varphi(\prod_{i\in P}p_i)\\ \nonumber &&
+ \sum_{l=2}^{\gamma_s}\frac {\varphi(n/p_s^{\gamma_s})}{\varphi(\prod_{i\in P}p_i)}\cdot
 p_s^{l-1} (p_s -1)\varphi(\prod_{i\in P}p_i)
\\ \nonumber && \left. 
+ \frac {\varphi(n/p_s^{\gamma_s})}{(p_s - 1)\varphi(\prod_{i\in P}p_i)}\cdot p_s^{\gamma_s} (p_s -
1) \varphi(\prod_{i\in P}p_i) \right )
\end{eqnarray}

If $\alpha_s = 1$ then $J_l=\emptyset$ for $l\geq 2$ and $\gamma_s=1$. Since the Euler totient
function is multiplicative, for $\alpha_s = 1$ we have $ \varphi (n) = (p_s - 1) \varphi (n / p_s)
$. Thus, the relation (\ref{eq:energy}) becomes
$$
E(X_n (1, p_s)) = 2^{k-1} \cdot \left(
\varphi(n)+\varphi(n/p_s)+\varphi(n)-(p_s - 1) \varphi(n/p_s)
\right) = 2^{k - 1} (\varphi (n) + \varphi (n / p_s)).
$$

If $\alpha_s=\gamma_s \geq 2$ then $J_{\gamma_{s+1}}=\emptyset$
since $\beta_s=\alpha_s-\gamma_s-1<0$ is not defined. Thus, the
relation (\ref{eq:energy}) is reduced to the first three summands as
follows
\begin{eqnarray}
E(X_n(1,p_s^{\gamma_s}))&=& 2^{k-1}\cdot \left (
(\varphi(n)+\varphi(n/p_s)) + (\varphi(n)-(p_s - 1)\varphi(n/p_s)) + (p_s -1)\varphi(n/p_s^{\gamma_s})\sum_{l=2}^{\gamma_s} p_s^{l-1} \right ) \nonumber\\
&=& 2^{k-1}\cdot \left (
2\varphi(n)+(p_s-2)\varphi(n/p_s^{\gamma_s}) + p_s(p_s^{\gamma_s-1} -1)\varphi(n/p_s^{\gamma_s}) \right ) \nonumber\\
&=& 2^{k-1} (2\varphi (n) + (p_s^{\gamma_s}-2p_s+2) \varphi (n / p_s)).\nonumber
\end{eqnarray}

If $\alpha_s>\gamma_s\geq 2$ the formula (\ref{eq:energy}) is composed of four summands, thus we have
\begin{eqnarray}
E(X_n(1,p_s^{\gamma_s}))&=&
2^{k-1} (2\varphi (n) + (p_s^{\gamma_s}-2p_s+2) \varphi (n / p_s)+p_s^{\gamma_s}\varphi(n/p_s^{\gamma_s}))\nonumber\\
&=& 2^{k} (\varphi (n) + (p_s^{\gamma_s}-p_s+1) \varphi (n / p_s)).\nonumber
\end{eqnarray}

This completes the proof.
\end{proof}

\begin{thm}
\label{thm:E(X_n(p,q))} Let $n\geq 4$ be an arbitrary integer. Then the energy of the integral
circulant graph $X_n(p,q)$ for $p=p_s$ and $q = p_t$, where $1 \leq s<t \leq k$, is given by
\begin{equation}
 \aligned E(X_n(p,q))=
    \left\{\begin{array}{rl} 2^{k} \varphi (n), & p \| n\ q \| n\\
                             3\cdot 2^{k-1} \varphi (n),&  2 \| n\ q^2 | n\\
                             2^{k-1} (2\varphi (n) +\varphi (n / p_t)\varphi(p_t)),&  p \| n\ q^2 | n\ p\neq 2\\
                             2^{k-1} (2\varphi (n) +\varphi (n / p_s)\varphi(p_s)), & p^2 | n\ q \| n\\
                             2^{k-1} (2\varphi (n) +\varphi (n / p_s)\varphi(p_s)+\varphi (n / p_t)\varphi(p_t)), & p^2 | n\ q^2 | n\\
                              \end{array}
                              \right.
\endaligned
\end{equation}
\end{thm}

\begin{proof}
Let $j = {p_1}^{\beta_1} p_2^{\beta_2} \cdot \ldots \cdot p_k^{\beta_k}\cdot J$ be a representation
of an arbitrary index $0\leq j\leq n-1$, where $gcd(J,n)=1$. The $j$-th eigenvalue of $X_n
(p_s,p_t)$ is given by
$$
\lambda_j=c(j,n/p_s)+c(j,n/p_t).
$$

Suppose that there exists prime number $p_i \mid j$ for some
$i \neq s,t$ such $\beta_i \leq \alpha_i - 2$. This implies that
$p_i^2\mid t_{n/p_s,j}$ and $p_i^2\mid t_{n/p_t,j}$. Furthermore, we have
$\mu(t_{n/p_t,j})=\mu(t_{n/p_s,j}) = 0$ and thus $\lambda_j = 0$.\smallskip

If $\beta_s\leq \alpha_s-3$ then $p_s^3 \mid t_{n/p_t,j}$ and $p_s^2 \mid t_{n/p_s,j}$. Similarly,
we conclude that $\lambda_j=0$.

If $\beta_t\leq \alpha_t-3$ then $p_t^3 \mid t_{n/p_s,j}$ and $p_t^2 \mid t_{n/p_t,j}$. Similarly,
we conclude that $\lambda_j=0$.
\smallskip

For an arbitrary index $j$, define the set $P=\{1\leq i\leq k\ |\
i \neq s,t, \beta_i=\alpha_i-1\}$.

Let $J_{l_1,l_2}=\{0\leq j\leq n-1\ |\ \beta_s=\alpha_s-l_1, \ \beta_t=\alpha_t-l_2,\ \alpha_i-1\leq
\beta_i\leq \alpha_i\ \mbox{ for}\ i \neq s,t\}$.

For $j\in J_{l_1,l_2}$, where $0\leq l_1,l_2\leq 2$,
we have
\begin{equation}
\label{eq:t_(n/p_s,j)}
t_{n/p_s,j}=\frac {n/p_s} {gcd(j,n/p_s)}=p_s^{\alpha_s-1-\min(\alpha_s-1,\alpha_s-l_1)}p_t^{\alpha_t-(\alpha_t-l_2)}\prod_{i\in P} p_i=
\left\{
\begin{array}{rl}
p_t^{l_2}\prod_{i\in P} p_i, &  0\leq l_1\leq 1  \\
p_sp_t^{l_2}\prod_{i\in P} p_i, & l_1= 2 \\
\end{array} \right..
\end{equation}

Similarly it follows
\begin{equation}
\label{eq:t_(n/p_t,j)}
t_{n/p_t,j}=
\left\{
\begin{array}{rl}
p_s^{l_1}\prod_{i\in P} p_i, &  0\leq l_2\leq 1  \\
p_tp_s^{l_1}\prod_{i\in P} p_i, & l_2= 2 \\
\end{array} \right..
\end{equation}

The number of indices  $j\in J_{l_1,l_2}$ with the same set $P$ is equal to
the number of $J$ such that
\begin{equation*}
\gcd \left (J, \frac n {p_1^{\beta_1}p_2^{\beta_2} \cdots p_s^{\alpha_s-l_1} \cdots p_t^{\alpha_t-l_2} \cdots
p_k^{\beta_k}} \right )=1.
\end{equation*}
The last equation implies that the number of such indices is equal to
\begin{equation}
\label{eq:number of j}
\varphi(\frac n {p_1^{\beta_1}p_2^{\beta_2} \cdots p_s^{\alpha_s-l_1} \cdots p_t^{\alpha_t-l_2} \cdots
p_k^{\beta_k}})=\varphi(p_s^{l_1}p_t^{l_2}\prod_{i\in P}p_i).
\end{equation}

Now, we distinguish four cases depending on the values of $l_1$ and $l_2$.

{\noindent \bf Case 1.} $0\leq l_1,l_2\leq 1.$

According to the relations (\ref{eq:t_(n/p_s,j)}) and (\ref{eq:t_(n/p_t,j)}) it follows
 $t_{n/p_s,j}=p_t^{l_2}\prod_{i\in P} p_i$ and $t_{n/p_t,j}=p_s^{l_1}\prod_{i\in P} p_i$ and therefore
the $j$-th eigenvalue is given by
\begin{eqnarray}
\lambda_j=c(j,n/p_s)+c(j,n/p_t)&=&(-1)^{|P|+l_2}\frac {\varphi(n/p_s)}{\varphi(p_t^{l_2}\prod_{i\in
P}p_i)}+(-1)^{|P|+l_1}\frac {\varphi(n/p_t)}{\varphi(p_s^{l_1}\prod_{i\in P}p_i)}\nonumber\\
&=&(-1)^{|P|}\frac {(-1)^{l_2}\varphi(n/p_s)\varphi(p_s^{l_1})+(-1)^{l_1}\varphi(n/p_t)\varphi(p_t^{l_2})}{\varphi(p_s^{l_1})\varphi(p_t^{l_2})\varphi(\prod_{i\in P}p_i)}.
\end{eqnarray}

If $l_1=l_2$ then
$$
|\lambda_j|=\frac {\varphi(n/p_s)\varphi(p_s^{l_1})+\varphi(n/p_t)\varphi(p_t^{l_2})}{\varphi(p_s^{l_1})\varphi(p_t^{l_2})\varphi(\prod_{i\in P}p_i)},
$$
while for $\l_1\neq l_2$ we have
$$
|\lambda_j|=\frac {(-1)^{l_2}\varphi(n/p_s)\varphi(p_s^{l_1})+(-1)^{l_1}\varphi(n/p_t)\varphi(p_t^{l_2})}{\varphi(p_s^{l_1})\varphi(p_t^{l_2})\varphi(\prod_{i\in P}p_i)},
$$
except for $p_s=2$, $p_t^2\mid n$ and $n\in 4\N+2$.\smallskip

It can be noticed that the numerator of the above relation for $l_1=0$ and $l_2=1$ is reduced to
\begin{equation}
\label{eq:l1=0,l2=1}
\varphi(n/p_t)\varphi(p_t)-\varphi(n/p_s),
\end{equation} while for $l_1=1$ and $l_2=0$ we have
\begin{equation}
\label{eq:l1=1,l2=0}
\varphi(n/p_s)\varphi(p_s)-\varphi(n/p_t).
\end{equation}

Since Euler totient function is multiplicative, for $\alpha_s = 1$ we have
$$
\varphi (n) = (p_s - 1) \varphi (n / p_s).
$$
Therefore, if $p_s\|n$ and $p_t\|n$ the above expressions are equivalent to
$\varphi(n)-\varphi(n/p_s)$ and $\varphi(n)-\varphi(n/p_t)$.\smallskip

Now assume that $p_s^2\mid n$ and $p_t^2\mid n$. We may conclude
that both expressions (\ref{eq:l1=0,l2=1}) and (\ref{eq:l1=1,l2=0})
are greater than zero if and only if $(p_s-1)(p_t-1)>1$. The last
relation is trivially satisfied. \smallskip

If $p_s^2\mid n$ and $p_t\|n$ then expression (\ref{eq:l1=0,l2=1})
is equivalent to $\varphi(n)-\varphi(n/p_s)$,
 which is greater than zero. Expression (\ref{eq:l1=1,l2=0}) is greater or equal to zero if and only if
$(p_s-1)(p_t-2)\geq 1$. This is true, since $p_t>p_s\geq 2$.\smallskip

If $p_s\| n$ and $p_t^2\mid n$ then expression (\ref{eq:l1=1,l2=0})
is equivalent to $\varphi(n)-\varphi(n/p_t)$,
 which is greater than zero. Expression (\ref{eq:l1=0,l2=1}) is greater or equal to zero if and only if
$(p_s-2)(p_t-1)\geq 1$. This is true, only if $p_s>2$. Therefore, for $p_s=2$, $p_t^2\mid n$ and $n\in 4\N+2$ we have that
\begin{equation}
\label{eq:special case 1}
|\lambda_j|=\frac {\varphi(n)-\varphi(n/p_t)\varphi(p_t)}{\varphi(p_s^{l_1})\varphi(p_t^{l_2})\varphi(\prod_{i\in P}p_i)},
\end{equation}
if $l_1=0$ and $l_2=1$, while
\begin{equation}
\label{eq:special case 2}
|\lambda_j|=\frac {\varphi(n)-\varphi(n/p_t)}{\varphi(p_s^{l_1})\varphi(p_t^{l_2})\varphi(\prod_{i\in P}p_i)},
\end{equation}
if $l_1=1$ and $l_2=0$.

The number of indices  $j\in J_{l_1,l_2}$ with the same set $P$ in all mentioned cases is given by (\ref{eq:number of j}) and equals
 $$\varphi(p_s^{l_1})\varphi(p_t^{l_2})\varphi(\prod_{i\in P}p_i).$$

{\noindent \bf Case 2.} $l_1=2$, $0\leq l_2\leq 1$.

According to relation (\ref{eq:t_(n/p_t,j)})
 we have that  $t_{n/p_t,j}=p_s^2\prod_{i\in P} p_i$, which further implies that $c(j,n/p_t)=0$. Now, using relation
(\ref{eq:t_(n/p_s,j)}) it holds that $t_{n/p_t,j}=p_sp_t^{l_2}\prod_{i\in P} p_i$ and thus
$$
\lambda_j=c(j,n/p_s)=(-1)^{|P|+l_2+1}\frac {\varphi(n/p_s)}{\varphi(p_s)\varphi(p_t^{l_2})\varphi(\prod_{i\in P}p_i)}.
$$
The number of indices  $j\in J_{l_1,l_2}$ with the same set $P$ is given by (\ref{eq:number of j}) and equals
 $$p_s\varphi(p_s)\varphi(p_t^{l_2})\varphi(\prod_{i\in P}p_i).$$

{\noindent \bf Case 3.} $0\leq l_1\leq 1$, $l_2=2$.

In this case we obtain symmetric expressions for $\lambda_j$ and the
number of indices with given set $P$.\medskip

{\noindent \bf Case 4.}  $l_1=l_2=2$.

According to the relations (\ref{eq:t_(n/p_s,j)}) and
(\ref{eq:t_(n/p_t,j)}) we have that $t_{n/p_s,j}=p_t^2\prod_{i\in P}
p_i$ and $t_{n/p_t,j}=p_s^2\prod_{i\in P} p_i$, which further
implies $\lambda_j=c(j,n/p_s)=c(j,n/p_t)=0$.
\bigskip

By summarizing all formulas in mention cases, the energy of $X_n(p_s, p_t)$ is given by
\begin{eqnarray}
\label{eq:energy1} E(X_n(p_s,p_t))&=&\sum_{j=0}^{n-1}|\lambda_j|\\ \nonumber
&=&\sum_{P\subseteq\{1,2,\ldots,k\}\setminus \{s,t\}} \Big ( (\varphi(n/p_s)+\varphi(n/p_t))+
(\varphi(n/p_s)\varphi(p_s)+\varphi(n/p_t)\varphi(p_t))\\ \nonumber &&
+ (\varphi(n/p_t)\varphi(p_t)-\varphi(n/p_s))+(\varphi(n/p_s)\varphi(p_s)-\varphi(n/p_t))\\
\nonumber && + 2\varphi(n/p_s)p_s \\ \nonumber && + 2\varphi(n/p_t)p_t \Big )
\end{eqnarray}

If $\alpha_s =\alpha_t=1$ then only nonempty sets are $J_{0,0}, J_{0,1},J_{1,0}$ and $J_{1,1}$.
Thus, the relation (\ref{eq:energy1}) becomes
\begin{eqnarray}
E(X_n (p_s, p_t)) &=& 2^{k-2} \cdot ( (\varphi(n/p_s)+\varphi(n/p_t))+(\varphi(n/p_s)\varphi(p_s)+\varphi(n/p_t)\varphi(p_t))\nonumber\\
&&+(\varphi(n/p_t)\varphi(p_t)-\varphi(n/p_s))+(\varphi(n/p_s)\varphi(p_s)-\varphi(n/p_t)) \nonumber\\
&=& 2^{k-2}(4\varphi(n))=2^{k}\varphi(n).
\end{eqnarray}

If $\alpha_s = 1$, $\alpha_t>1$ and $p_s \neq 2$  then $J_{2,0}, J_{2,1}$ and $J_{2,2}$ are the empty sets. Also, for $\alpha_t > 1$
we have
$$
\varphi (n) = (p_t - 1) p_t^{\alpha_t - 1} \varphi (n / {p_t}^{\alpha_t}) =
p_t \varphi (p_t^{\alpha_t - 1}) \varphi (n / {p_t}^{\alpha_t}) = p_t \varphi (n / p_t) > (p_t - 1) \varphi (n / {p_t}).
$$
Therefore, from the relation (\ref{eq:energy1}) follows
\begin{equation}
\label{eq:symetric}
E(X_n(p_s,p_t))=2^{k-2}\cdot \left(
2(\varphi(n)+\varphi(n/p_t)\varphi(p_t))+ 2\varphi(n/p_t)p_t\right )
= 2^{k-1} (2\varphi (n) +\varphi (n / p_t)\varphi(p_t)).
\end{equation}

If $\alpha_s = 1$, $\alpha_t>1$ and $p_s = 2$, according to relations (\ref{eq:special case 1}) and (\ref{eq:special case 2})
the energy is equal to

\begin{eqnarray}
E(X_n(p_s,p_t))&=&2^{k-2}\cdot
(\varphi(n/p_s)+\varphi(n/p_t))+(\varphi(n/p_s)\varphi(p_s)+\varphi(n/p_t)\varphi(p_t))\\
&&+ (\varphi(n)-\varphi(n/p_t)\varphi(p_t))+(\varphi(n)-\varphi(n/p_t))\nonumber\\
&&+ 2p_t\varphi(p_t)\nonumber\\
&=& 2^{k-1} (2\varphi (n) +\varphi (n / p_t)p_t)=3 \cdot 2^{k-1} \varphi (n)\nonumber.
\end{eqnarray}

If $\alpha_s > 1$ and $\alpha_t=1$, we have similar equation as in the previous case:

$$
E(X_n(p_s,p_t))= 2^{k-1} (2\varphi (n) +\varphi (n / p_s)\varphi(p_s)).
$$

If $\alpha_s > 1$ and $\alpha_t>1$  then all sets $J_{l_1,l_2}$, for $0\leq l_1,l_2\leq 2$
 are nonempty and thus the energy is equal to
\begin{eqnarray}
E(X_n(p_s,p_t))&=&2^{k-2}\cdot \left(
2(\varphi(n/p_s)\varphi(p_s)+\varphi(n/p_t)\varphi(p_t))+ 2\varphi(n/p_s)p_s +2\varphi(n/p_t)p_t\right )\\
&=& 2^{k-1} (2\varphi (n) +\varphi(n/p_s)p_s+\varphi (n / p_t)\varphi(p_t)).\nonumber
\end{eqnarray}

This completes the proof.
\end{proof}

\section{Classes of non-cospectral graphs with equal energy}

Let $n=p_1p_2\ldots p_s  p_{s+1}^{\alpha_{s+1}}\ldots p_{k}^{\alpha_{k}}$ be a prime factorization of $n$,
where $\alpha_i\geq 2$ for $s+1\leq i\leq k$.
Using the result of Theorem (\ref{thm:E(X_n(p,q))}) we see that the energy of integral
circulant graph $X_n (p_i, p_j)$ does not depend on the choice of $p_i$ and $p_j$, if $p_i,p_j\|n$.
Also, the same conclusion can be derived if we consider the graphs $X_n (2, p_j)$ for $\alpha_j\geq 2$ and $n\in 4\N+2$.

Since the order of the graph $X_n (p_i, p_j)$ is equal to $\varphi (n / p_i) + \varphi (n / p_j)$, which is at the
same time the largest eigenvalues also, we can construct at least $s+1$ non-cospectral regular $n$-vertex hyperenergetic
graphs,
$$
X_n (1), \ X_n (p_1, p_2), \ X_n (p_1, p_3), \ \ldots, \ X_n (p_1, p_s),
$$
with equal energy. Similarly, we obtain the second class of $k-s$ non-cospectral  graphs with equal energy.

$$
X_n (2, p_{s+1}), \ X_n (2, p_{s+2}), \ \ldots, \ X_n (2, p_k),
$$

Moreover, we can consider a square-free number $n = p_1 p_2
\cdot \ldots \cdot p_k$ and prove that the following $\binom{k}{2}$ graphs
$$
X_n (p_1, p_2), \ X_n (p_1, p_3), \ \ldots, \ X_n (p_{k - 1}, p_k),
$$
are non-cospectral.



Consider the integral circulant graph $X_n (p_i, p_j)$. The largest eigenvalue and the degree of
$X_n (p_i, p_j)$ is $\varphi (n / p_i) + \varphi (n / p_j)$. According to the proof of Theorem 4.2
from \cite{Il09}, the second largest value among $|\lambda_1|, |\lambda_2|, \ldots,
|\lambda_{n-1}|$ equals
\begin{eqnarray}
\label{eq:s (X_n (p_i, p_j))}
s (X_n (p_i, p_j)) &=& \max \left \{ \frac{\varphi(\frac{n}{p_i}) +
\varphi(\frac{n}{p_j})}{\varphi(p)}, \frac{\varphi(n) - \varphi(\frac{n}{p_i})}{\varphi(p_j)},
\frac{\varphi(n) - \varphi(\frac{n}{p_j})}{\varphi(p_i)}, \frac{2\varphi(n)}{\varphi(p_i p_j)}
\right \} \nonumber\\
&=& \varphi \left (\frac{n}{p_i p_j} \right) \cdot \max \left \{ \frac{p_i + p_j - 2}{\varphi
(p_{ij})}, p_i - 2, p_j - 2, 2 \right \},
\end{eqnarray}
where $p_{ij}$ denotes the smallest prime number dividing $\frac{n}{p_i p_j}$. \smallskip

Assume that graphs $X_n (p_i, p_j)$ and $X_n (p_r, p_q)$ are cospectral, with $p_j > p_i$ and $p_q
> p_r$. Furthermore, assume that $p_i > p_r$.\medskip

\noindent {\bf Case 1.} $p_i > 3$ and $p_r > 3$.

From $p_j > p_i > 3$ it easily follows that
$$
s (X_n (p_i, p_j)) = \varphi \left (\frac{n}{p_i p_j} \right) \cdot (p_j - 2).
$$
By equating the largest eigenvalues of these graphs and the values $s (X_n (p_i, p_j))$
 and $s (X_n (p_r, p_q))$, it follows
\begin{equation}
\label{eq:largest eigenvalues}
\varphi ( p_r p_q ) \cdot (p_i + p_j - 2) = \varphi ( p_i p_j ) \cdot (p_r + p_q - 2)
\end{equation}
\begin{equation}
\label{eq:s ()}
\varphi ( p_r p_q ) \cdot (p_j - 2) = \varphi ( p_i p_j ) \cdot (p_q - 2).
\end{equation}
Notice that we used the multiplicative property of the Euler function.

By subtraction, we get
\begin{equation}
\label{eq:subtraction}
(p_r - 1)(p_q - 1) \cdot p_i = (p_i - 1)(p_j - 1) \cdot p_r.
\end{equation}
Assume without loss of generality that $p_i < p_r$. It follows that $p_i \mid p_j - 1$ and $p_r
\mid p_q - 1$. Since $p_i \mid \varphi(p_j)$ and $p_r
\mid \varphi(p_q)$, from the relation (\ref{eq:largest eigenvalues}),
we conclude that $p_i \mid \varphi ( p_r p_q )$ and $p_r
\mid \varphi ( p_i p_j )$. Since $p_i < p_r$, we have $p_r \mid p_j - 1$ and from the
relation (\ref{eq:subtraction}) it holds that $p_r^2 \mid p_q - 1$. Similarly, from the relation (\ref{eq:largest eigenvalues})
it follows that $p_r^2 \mid p_j -1$ and again according to (\ref{eq:largest eigenvalues}) $p_r^3 \mid p_q - 1$ holds.
Using infinite descent, we get that four-tuple $(p_i, p_j,p_r, p_q)$ does not exist.
\medskip

\noindent {\bf Case 2.} $p_i > 3$ and $p_r=3$.

We distinguish two cases depending on the values of $p_{rq}$. Let
$p_{rq}=2$. Then, according to the relation (\ref{eq:s (X_n (p_i,
p_j))}) we have that
$$
s (X_n (p_r, p_q)) = \varphi \left (\frac{n}{3 p_q} \right) \cdot (p_q + 1).
$$
By equating the largest eigenvalues of these graphs and the values $s (X_n (p_i, p_j))$
and $s (X_n (p_r, p_q))$, it follows
\begin{equation*}
\varphi ( 3 p_q ) \cdot (p_i + p_j - 2) = \varphi ( p_i p_j ) \cdot (p_q + 1)
\end{equation*}
\begin{equation*}
\varphi ( 3 p_q ) \cdot (p_j - 2) = \varphi ( p_i p_j ) \cdot (p_q + 1).
\end{equation*}

The last two equation hold only if $p_i + p_j - 2=p_j - 2$ which a contradiction.
\smallskip

Let $p_{rq}>2$. Since $p_q>p_r=3$ and therefore $p_q\geq 5$, we have that
$$
p_q-2\geq \frac{p_q+1} {2} \geq \frac{p_q+1} {\varphi(p_{rq})}.
$$
From the last relation we conclude that
$$
s (X_n (p_r, p_q)) = \varphi \left (\frac{n}{3 p_q} \right) \cdot \frac {p_q + 1}{\varphi(p_{rq})}.
$$
By equating the largest eigenvalues of these graphs and the values $s (X_n (p_i, p_j))$
and $s (X_n (p_r, p_q))$, it follows
\begin{equation*}
\varphi ( 3 p_q ) \cdot (p_i + p_j - 2) = \varphi ( p_i p_j ) \cdot (p_q + 1)
\end{equation*}
\begin{equation*}
\varphi ( 3 p_q )\cdot \varphi(p_{rq}) \cdot (p_j - 2) = \varphi ( p_i p_j ) \cdot (p_q + 1).
\end{equation*}
From the last relations we see that $p_i + p_j - 2=\varphi(p_{rq}) \cdot (p_j - 2)$ holds. Next, it
holds that $p_i\leq p_j-2$, which further implies $\varphi(p_{rq}) \cdot (p_j - 2)\leq 2(p_j-2)$.
But that is only the case if $\varphi(p_{rq})\leq 2$ or equivalently $p_{rq}\leq 3$, which is a
contradiction.

\medskip
\noindent {\bf Case 3.} $p_i > 3$ and $p_r=2$.

We distinguish two cases depending on the values of $p_q$.

Let $p_q=3$. From the relation (\ref{eq:s (X_n (p_i, p_j))}) it can be concluded that
$$
s (X_n (p_r, p_q)) = 2\cdot \varphi \left (\frac{n}{6} \right).
$$
By equating the largest eigenvalues of these graphs and the values $s (X_n (p_i, p_j))$
and $s (X_n (p_r, p_q))$, it follows
\begin{equation}
\label{eq:Case 1.3}
\varphi (6) \cdot (p_i + p_j - 2) = 3\cdot\varphi ( p_i p_j )
\end{equation}
\begin{equation*}
\varphi (6) \cdot (p_j - 2) = 2\cdot\varphi ( p_i p_j )
\end{equation*}
By subtraction, we get
$$
\varphi (6) \cdot p_i=\varphi ( p_i p_j )=(p_i-1)(p_j-1).
$$
From the last relation it holds that $p_i\mid p_j-1$ and combining with the relation (\ref{eq:Case
1.3}) we obtain that $p_i\mid p_j-2$. This is a contradiction, since $p_j-2$ and $p_j-2$ are
relatively prime.
\smallskip

Let $p_q>3$. Since the following inequality holds
$$
p_q-2\geq \frac{p_q} {2} \geq \frac{p_q} {\varphi(p_{rq})},
$$
we have
$$
s (X_n (p_r, p_q)) = \varphi \left (\frac{n}{3 p_q} \right) \cdot (p_q - 2).
$$
Now, this case is reduced to the equations (\ref{eq:largest eigenvalues}) and (\ref{eq:s ()}) from
Case 1, where we obtained a contradiction.

\medskip

\noindent {\bf Case 4.} $p_i=3$ and $p_r = 2$.

Since $p_q\neq p_i$ we have $p_q\geq 5$, which further implies
$\max\{p_q/\varphi(p_{rq}),p_q-2,2)\}=p_q-2$. Therefore, it holds
that
$$
s (X_n (p_r, p_q)) =  \varphi \left (\frac{n}{2p_q} \right)(p_q-2).
$$
Moreover, as $p_j\neq p_r$ and $p_r=2$, we obtain $p_{ij}=2$. Thus, we conclude
$$
\max\{\frac {p_j + 1}{\varphi(p_{ij})},1,p_j-2,2)\}=p_j+1
$$
and
$$
s (X_n (p_i, p_j)) =  \varphi \left (\frac{n}{3p_j} \right)(p_j+1).
$$
 By equating the largest eigenvalues of these graphs and the values $s (X_n (p_i, p_j))$
 and $s (X_n (p_r, p_q))$, it follows
\begin{equation}
\varphi ( 2 p_q ) \cdot (p_j + 1) = \varphi ( 3 p_j ) \cdot  p_q
\end{equation}
\begin{equation}
\varphi ( 2 p_q ) \cdot (p_j + 1) = \varphi ( 3 p_j ) \cdot (p_q - 2).
\end{equation}
From the previous relations we trivially get that four-tuple $(p_i, p_j,p_r, p_q)$ does not exist in this case.

\medskip

This way we actually prove that two cospectral integral circulant
graphs $\ICG_n(D_1)$ and $\ICG_n(D_2)$ must be isomorphic i.e.
$D_1=D_2$, for a square-free number $n$ and two-element divisor sets
$D_1$ and $D_2$ containing prime divisors. Therefore, we support
conjecture proposed by So \cite{so06}, that two graphs $\ICG_n(D_1)$
and $\ICG_n(D_2)$ are cospectral if and only if $D_1=D_2$. The
conjecture was only proven for the trivial cases where $n$ being
square-free and product of two primes. Our result is obviously one
form of generalization.

\section{Concluding remarks}

In this paper we focus on some global characteristics of the energy
of integral circulant graphs such as energy modulo four and
existence of non-cospectral graphs classes with equal energy. We
also find explicit formulas for the energy of $\ICG_n(D)$ classes
with two-element set $D$. In contrast to \cite{Il09}, the
calculation of these formulas require extensive discussion in many
different cases. Some further generalizations on this topic would
require much more case analysis. The examples of such
generalizations are calculating the energy of the graphs with three
or more divisors, graphs with square-free orders etc. The general
problem of calculating the energy of $\ICG_n(D)$ graphs seems very
difficult, since as we increase the number of divisors in $D$ we
have more sign changes in Ramanujan functions $c (n, i)$.

For the further research we also propose some new general characteristics of the energy such as
studying minimal and maximal energies for a given integral circulant graph, and characterizing the
extremal graphs. We will use the following nice result from \cite{GuFiPeRa07,ZhGuPeRaMe07}

\begin{thm}
Let G be a regular graph on $n$ vertices of degree $r > 0$. Then
$$
E (G) \geq n,
$$
with equality if and only if every component of $G$ is isomorphic to the complete bipartite graph
$K_{r,r}$.
\end{thm}

The proof is based on the estimation
$$
E (G) \geq \frac{M_2^2}{\sqrt{M_2 M_4}},
$$
where $M_2 = 2m$ and $M_4$ are spectral moments of graph $G$, defined as
$$
M_k = \sum_{i = 1}^n \lambda_i^k.
$$
The fourth moment is equal to $M_4 = 8q - 2m + 2 \sum_{v \in V} deg^2 (v)$, where $q$ is the number
of quadrangles in~$G$.

Let $n$ be even number and assume that $\ICG_n (D^*)$ is isomorphic to $K_{n/2, n/2}$. The present
authors in \cite{BaIl09} proved the following

\begin{thm}
\label{d components} Let $d_1, d_2, \ldots, d_k$ be divisors of $n$ such that the greatest common
divisor $\gcd (d_1, d_2, \ldots, d_k)$ equals $d$. Then the graph $\ICG_n (d_1, d_2, \ldots, d_k)$
has exactly $d$ connected components isomorphic to $\ICG_{n / d} (\frac{d_1}{d}, \frac{d_2}{d},
\ldots, \frac{d_k}{d})$.
\end{thm}

In this case the complement of $\ICG_n (D^*)$, denoted by $\ICG_n (\overline{D})$, must contain
exactly two connected components that are cliques, and for $\overline{D} = \{ d_1, d_2, \ldots, d_k
\}$ we have $\gcd (d_1, d_2, \ldots, d_k) = 2$ and $\ICG_{n / 2} (\frac{d_1}{2}, \frac{d_2}{2},
\ldots, \frac{d_k}{2})$ is isomorphic to a complete graph $K_{n/2}$. It simply follows that the set
$\overline{D}$ must contain all even divisors of $n$ and therefore $D^*$ is the set of all odd
divisors of $n$. Therefore, the degree of $\ICG_n (D^*)$ is equal to $\frac{n}{2} = \sum_{d \in
D^*} \varphi (\frac{n}{d})$ and $\ICG_n (D^*)$ is isomorphic to a complete bipartite graph $K_{n/2,
n/2}$. Recall that the spectra of the complete bipartite graph $K_{m, n}$ consists of $\sqrt{mn}$,
$-\sqrt{mn}$ and $0$ with multiplicity $n - 2$. It follows that $|\lambda_{n / 2}| = |\lambda_0| =
\frac{n}{2}$ and for $k \neq 0, \frac{n}{2}$ we have the following nice identity
\begin{eqnarray*}
\lambda_k &=& \sum_{d \mid n, \ d \ odd} c \left(k, d\right) \\
&=& \sum_{d \mid n, \ d \ odd} \mu \left ( \frac{d}{\gcd (k, d)} \right ) \cdot \frac{ \varphi (d)
} {\varphi \left (\frac{d}{\gcd (k, d)} \right )} \\
&=& 0.
\end{eqnarray*}

Using computer search, for odd $n$ the minimum is $2 n (1 - \frac{1}{p})$, where $p$ is the
smallest prime dividing $n$. The extremal integral circulant graph contains all divisors of $n$
that are not divisible by $p$ (and the complement of such graph is composed of $p$ cliques). We
leave this observation as a conjecture.

\vspace{0.5cm} {\bf Acknowledgement. } The authors gratefully
acknowledge support from Research projects 174010 and 174033 of the
Serbian Ministry of Science.


\begin{thebibliography}{99}

\bibitem{ahmadi}
A. Ahmadi, R. Belk, C. Tamon and C. Wendler, \textit{On mixing of continuous time quantum walks on
some circulant graphs}, Quant. Inform. Comput. 3 (2003) 611–-618.

\bibitem{sandiego} R.J. Angeles-Canul, R.M. Norton, M.C. Opperman,
C.C. Paribello, M.C. Russell, C. Tamonk, \textit{Perfect state transfer, integral ciculants and
join of graphs}, Quant. Inform. Comput. 10 (2010) 325–-342.

\bibitem{sandiego1} R.J. Angeles-Canul, R.M. Norton, M.C. Opperman,
C.C. Paribello, M.C. Russell, C. Tamonk, \textit{Quantum perfect state transfer on weighted join
graphs}, Int. J. Quantum Inf. 7 (2009), 1429--1445.

\bibitem{AkMoZa09}
    S. Akbari, F. Moazami, S. Zare, \textit{Kneser Graphs and their
    Complements are Hyperenergetic}, MATCH Commun. Math. Comput. Chem. 61 (2009), 361--368.

\bibitem{BaPa04}
    R. B. Bapat, S. Pati, \textit{Energy of a graph is never an odd integer}, Bull. Kerala Math. Assoc. 1
    (2004) 129--132.

\bibitem{BaIl09}
    M. Ba\v si\' c, A. Ili\' c, \textit{On the clique number of integral circulant graphs},
    Appl. Math. Lett. 22 (2009) 1406--1411.

\bibitem{stevanovic08a}
    M. Ba\v si\'c, M. Petkovi\'c, D. Stevanovi\'c, \textit{Perfect state transfer in integral circulant graphs},
    Appl. Math. Lett. 22 (2009) 1117--1121.


\bibitem{Ba10}
    M. Ba\v si\'c, \textit{Characterization of circulant graphs
having perfect state transfer}, manuscript, 2010.

\bibitem{BlSh08}
    S. Blackburn, I. Shparlinski, \textit{On the average energy of circulant graphs},
    Linear Algebra Appl. 428 (2008), 1956--1963.

\bibitem{BoViAb08}
    A. S. Bonif\' acio, C. T. M. Vinagre, N. M. M. de Abreu,
    \textit{Constructing pairs of equienergetic and non-cospectral graphs},
    Appl. Math. Lett. 21 (2008), 338--341.

\bibitem{BrStGu04}
    V. Brankov, D. Stevanovi\' c, I. Gutman,
    \textit{Equienergetic chemical trees},
    J. Serb. Chem. Soc. 69 (2004), 549--553.


\bibitem{fizicarski1} M. Christandl, N. Datta, T.C. Dorlas, A. Ekert,
A. Kay, and A.J. Landahl, \textit{Perfect transfer of arbitrary states in quantum spin networks},
Phys. Rev. A 71:032312, 2005.

\bibitem{Da79}
    P. J. Davis,
    \textit{Circulant matrices},
    Pure and Applied Mathematics, John Wiley \& Sons, New York-Chichester-Brisbane, 1979.

\bibitem{Go08}
    C. D. Godsil, \textit{Periodic Graphs}, arXiv:0806.2074v1 [math.CO], 12 Jun 2008.

\bibitem{Gu78}
    I. Gutman, \textit{The energy of a graph}, Ber. Math. Stat. Sekt.
    Forschungszent. Graz 103 (1978) 1--22.

\bibitem{Gu01}
    I. Gutman, The energy of a graph: Old and new
        results, in: A. Betten, A. Kohnert, R. Laue, A. Wassermann
        (Eds.), Algebraic Combinatorics and Applications,
        Springer-Verlag, Berlin, 2001, pp. 196--211.

\bibitem{Gu99}
    I. Gutman, \textit{Hyperenergetic molecular graphs}, J. Serb. Chem. Soc. 64
    (1999) 199--205.

\bibitem{GuFiPeRa07}
    I. Gutman, S. Z. Firoozabadi, J. A. de la Pe\~{n}a, J. Rada, \textit{On the energy of regular graphs},
    MATCH Commun. Math. Comput. Chem. 57 (2007) 435--442.

    \bibitem{E} I. Gutman, X. Li, J. Zhang, Graph energy, in:
        M. Dehmer, F. Emmert-Streib (Eds.), Analysis of Complex
        Networks. From Biology to Linguistics, Wiley-VCH,
        Weinheim, 2009, pp. 145--174.



\bibitem{HaWr80}
    G. H. Hardy, E. M. Wright,
    \textit{An Introduction to the Theory of Numbers},
    Fifth edition, Oxford University Press, New York, 1980.

\bibitem{hwang03}
    F. K. Hwang, \textit{A survey on multi-loop networks}, Theoretical Computer Science 299
    (2003) 107--121.

\bibitem{Il09}
    A. Ili\' c,
    \textit{The energy of unitary Cayley graphs},
    Linear Algebra Appl. 431 (2009) 1881--1889.

\bibitem{IlBa09}
    A. Ili\' c, M. Ba\v si\' c, \textit{On the chromatic number of integral circulant graphs},
    Comput. Math. Appl. 60 (2010) 144-150.

\bibitem{IlBaGu10}
    A. Ili\' c, M. Ba\v si\' c, I. Gutman,
    \textit{Triply Equienergetic Graphs},
    MATCH Commun. Math. Comput. Chem. 64 (2010) 189--200.

\bibitem{Il10}
    A. Ili\' c,
    \textit{Distance spectra and distance energy of integral circulant graphs},
    Linear Algebra Appl. 433 (2010) 1005--1014.

\bibitem{klotz07}
    W. Klotz, T. Sander, \textit{Some properties of unitary Cayley graphs},
    Electron. J. Combin. 14 (2007) \#R45

\bibitem{C}
J. Liu, B. Liu, E-L equienergetic graphs, MATCH Commun. Math. Comput.
        Chem. 66 (2011) 971-976.


\bibitem{B} O. Miljkovi\'c, B. Furtula, S. Radenkovi\'c, I. Gutman, Equienergetic
        and almost-equienergetic trees, MATCH Commun. Math. Comput. Chem. 61 (2009)
        451-461.


\bibitem{PiGu08}
    S. Pirzada, I. Gutman,
    \textit{Energy of a graph is never the square root of an odd integer},
    Appl. Anal. Discrete Math. 2 (2008) 118--121.

\bibitem{RaWa07}
    H. S. Ramane, H. B. Walikar,
    \textit{Construction of equienergetic graphs},
    MATCH Commun. Math. Comput. Chem. 57 (2007), 203--210.

\bibitem{RaVe09}
    H. N. Ramaswamy, C. R. Veena,
    \textit{On the Energy of Unitary Cayley Graphs},
    Electron. J. Combin. 16 (2007) \#N24


\bibitem{D} O. Rojo, L. Medina, Constructing graphs with energy $\sqrt{r}\,E(G)$
        where $G$ is a bipartite graph, MATCH Commun. Math. Comput. Chem. 62 (2009)
        465-472.


\bibitem{SaSeSh07}
    N. Saxena, S. Severini, I. Shparlinski, \textit{Parameters of integral circulant graphs and periodic quantum dynamics},
    Int. J. Quant. Inf. 5 (2007), 417--430.


\bibitem{so06}
    W. So, \textit{Integral circulant graphs}, Discrete Math. 306 (2006)
    153--158.

\bibitem{So09}
    W. So, \textit{Remarks on some graphs with large number of
    edges}, MATCH Commun. Math. Comput. Chem. 61 (2009), 351--359.

\bibitem{A} I. Stankovi\'c, M. Milo\v{s}evi\'c, D. Stevanovi\'c, Small and not so
        small equienergetic graphs, MATCH Commun. Math. Comput. Chem. 61 (2009)
        443-450.

\bibitem{ZhGuPeRaMe07}
    B. Zhou, I. Gutman, J. A. de la Pe\~ na, J. Rada, L. Mendoza,
    \textit{On spectral moments and energy of graphs},
    MATCH Commun. Math. Comput. Chem. 57 (2007) 183--191.

\bibitem{Wiki}
    \url{http://en.wikipedia.org/wiki/Ramanujan's_sum}

\end{thebibliography}
\end{document}